\documentclass[journal]{IEEEtran}
\ifCLASSINFOpdf
\else
\fi
\usepackage{hyperref}

 \usepackage[pdftex]{graphicx}
 \usepackage{color}

\usepackage{amsmath}
\usepackage{amsthm}
\usepackage{amssymb}
\usepackage[ruled,vlined,linesnumbered]{algorithm2e}
\SetKwProg{Init}{Init:}{}{}

\newcommand{\leaveout}[1]{}

\makeatletter
\renewcommand{\p@enumii}{}
\makeatother

\newcommand{\cE}{{\mathcal E}}
\newcommand{\cF}{{\mathcal F}}

\newcommand{\cI}{{\mathcal I}}

\newcommand{\cL}{{\mathcal L}}

\newcommand{\cS}{{\mathcal S}}
\newcommand{\cT}{{\mathcal T}}

\newcommand{\cW}{{\mathcal W}}

\newcommand{\bx}{{\mathbf x}}
\newcommand{\bz}{{\mathbf z}}
\newcommand{\bu}{{\mathbf u}}

\newcommand{\blo}{{\cL}}

\newcommand\R{{\mathbb R}}

\newcommand{\dA}{\mathbf A}
\newcommand{\dB}{\mathbf B}

\newcommand{\dP}{\mathbf P}

\newcommand{\dom}[1]{\mathrm{dom}\left(#1\right)}
\newcommand{\range}[1]{\mathrm{ran}\left(#1\right)}

\newcommand{\Ker}[1]{\ker\left(#1\right)}

\newcommand{\ipdp}[2]{\langle #1 , #2 \rangle}
\newcommand{\Ipdp}[2]{\left\langle #1 , #2 \right\rangle}

\newcommand{\set}[1]{\left\lbrace #1 \right\rbrace}

\newcommand{\bi}{\begin{itemize}}
\newcommand{\ei}{\end{itemize}}
\newcommand{\be}{\begin{enumerate}}
\newcommand{\ee}{\end{enumerate}}

\newcommand{\sbm}[1]{\left[\begin{smallmatrix}#1
\end{smallmatrix}\right]}

\newcommand{\bbm}[1]{\begin{bmatrix}#1\end{bmatrix}}

\newtheorem{lemma}{Lemma}
\newtheorem{theorem}[lemma]{Theorem}

\theoremstyle{definition}

\newtheorem{remark}[lemma]{Remark}

\begin{document}

\title{Region-free explicit model predictive control for
	linear systems on Hilbert spaces}

\author{
Mikael~Kurula$^1$, 
Jukka-Pekka Humaloja$^{2,*}$ and 
Stevan~Dubljevic$^2$
\thanks{$^1$M. Kurula is with \AA bo Akademi University, \AA bo, Finland. e-mail: mkurula@abo.fi} 
\thanks{$^2$S. Dubljevic and J.-P. Humaloja are with University of Alberta, Edmonton, Canada. 
e-mail: stevan.dubljevic@ualberta.ca, jphumaloja@ualberta.ca}
\thanks{$^*$J.-P. Humaloja is funded by grants from the Jenny and Antti Wihuri Foundation 
and the Vilho, Yrj\"o and Kalle V\"ais\"al\"a Foundation}
\thanks{Manuscript received X; revised Y.}}

\maketitle

\begin{abstract}
We extend discrete-time explicit model predictive control (MPC) rigorously to linear distributed parameter systems. After formulating an MPC framework and giving a relevant KKT theorem, we realize fast regionless explicit MPC by using the dual active set method QPKWIK. A Timoshenko beam with input and state constraints is used to demonstrate the efficacy of the design at controlling a continuous-time hyperbolic PDE with constraints, using a discrete-time explicit MPC controller. 
\end{abstract}


\section{Introduction}

\emph{Model predictive control} (MPC for short, also called \emph{receding horizon control}) is a technique for approximately solving an 
infinite-horizon optimal control problem by instead solving a sequence of finite-horizon problems. This gives a near-optimal control signal implicitly, formed by the initial parts of the solutions of a sequence of optimization problems. In \cite{BMDP02} and some other papers of the same time, such as \cite{JPS02} and \cite{SdDG00}, it was realized that these optimization problems can sometimes be solved explicitly, foreseeing major speedups compared to traditional MPC using naive online optimization. These seminal papers sparked intensive research on what is now called \emph{explicit model 
predictive control (explicit MPC)}. Explicit MPC classically consists of two phases, the \emph{off-line phase}, where the piecewise affine optimal control step is precalculated and stored in advance via the process of \emph{state space exploration} (see the next paragraph), and the \emph{on-line} phase, where at every time step, \emph{point location} is performed in order to 
determine and retrieve the optimal control step. 

Automatic state space exploration, where the state space of the plant is divided into polygonal regions, on which the optimal control signal is piecewise affine, is tricky already in the finite-dimensional setting, due to degeneracy issues; see \cite{SKJTJ06}. Algorithms that work around degeneracy have been proposed, e.g., in \cite{PaSa10,ODP17}. The complexity of the explicit control increases rapidly with both the state-space dimension of the plant and the prediction horizon, leading to challenges with storing the explicit control in a data structure that facilitates retrieving the optimal control \cite{ZhXi18}. Kvasnica and coauthors have studied complexity reduction in explicit MPC \cite{KJPHKB19}.

Later, a much lighter explicit MPC technique appeared in \cite{GBN11}, the so called \emph{region-free} approach; see also \cite[\S2.2]{KJPHKB19} and the references therein. The region-free approach avoids state exploration and storage of the full piecewise affine controller, hence eliminating the off-line phase altogether. Instead, in the on-line phase, it calculates the optimal control step on the fly, using the set of constraints active at the optimal control step associated to the present plant state. Such methods are referred to as \emph{active set methods}, and prime examples are the dual active set method by Goldfarb and Idnani \cite{GoId83} and its refined version QPKWIK \cite{ScBi94}. Recent work of M\"onnigmann and coauthors \cite{Mon19,MiMoe20,JPM17} investigates the dynamics of the optimal active set and the explicit MPC controller, as functions of the horizon length. 

While \emph{explicit} MPC is introduced here as being new for infinite-dimensional systems, generic MPC theory for infinite-dimensional discrete-time (mainly nonlinear) plants has been developed, e.g., by Gr\"une and coauthors; see \cite{GrPaBook}. In his PhD thesis \cite{AltmullerThesis}, Altm\"uller demonstrates that some of the central ideas in \cite{GrPaBook} are effective also for hyperbolic PDEs. In this technical note, we concentrate on how to compute the MPC control explicitly and efficiently, for affinely constrained linear systems with quadratic cost, rather than on fundamental properties of MPC itself, such as stabilization or feasibility. 

We contribute a simple, fast, light and clean setup for \emph{explicit} MPC of distributed parameter systems, which avoids all degeneracy issues that clutter the present literature. The 
proposed methodology is applied to a Timoshenko beam model with input and state constraints, 
where we show how one can control the continuous-time plant with the discretely generated
control signal. While we focus on infinite-dimensional systems, we point out that the paper 
provides a powerful implementation for the finite-dimensional literature as well. A preliminary 
version of the proposed approach was presented without proofs in \cite{DHK22MTNS}.

MPC for infinite-dimensional continuous-time plants, usually without constraints, has been studied before; 
see in particular \cite{AK18}, but also, e.g., \cite{ItKu02,PGB14, AzKu19,KuPf20}. The 
continuous-time setting in these studies makes the analysis rather technical, and hoping to make 
an impact in engineering we will prefer a more practically oriented and elementary approach in 
discrete time, which allows us to handle a finite number of general affine constraints, similar to  \cite{XuDu17, DuHu20}.

The paper is laid out in the following way: In \S\ref{sec:MPC}, we describe a general and flexible MPC setup, which we further reformulate as a standard parametric Quadratic Program (pQP) in \S\ref{sec:pQP}. In \S\ref{sec:expsol}, we solve the pQP explicitly and make a connection to QPKWIK, thus completing the implementation of reliable explicit MPC for linear PDE systems. Finally, the paper is ended in \S\ref{sec:example}, where the setup is demonstrated on the Timoshenko beam.

\section{The discrete-time MPC formulation}\label{sec:MPC}

In this section, we describe a rather flexible MPC formulation: We allow time-dependent 
weights, 
time-dependent constraints, and a cost on the rate of change of control. We note that, similar to 
\cite[\S6]{BMDP02}, the 
basic formulation can easily be extended in several ways, in order to account, e.g., for reference tracking, 
measured disturbances, and soft constraints, but we omit those here for the sake of brevity.

Let $U$ and $X$ be \emph{real} Hilbert spaces and let $A\in\cL(X)$, $B\in\cL(U,X)$ be bounded linear operators. Consider the discrete-time plant dynamics
\begin{equation}\label{eq:discrPlant}
	x_{n+1}=Ax_n+Bu_n, \quad n\in\mathbb{N}_0,
\end{equation}
where $u_n\in U$ is the \emph{input} at the discrete time step $n$ and $x_n\in X$ is the \emph{state} at time $n$. The objective is to steer $x_n$ in \eqref{eq:discrPlant} to zero in such a way that a quadratic cost functional of the following form is minimized:
\begin{equation}\label{eq:infhoriz}
\begin{aligned}
	&\sum_{n=0}^\infty \bigg(\left\langle\bbm{Q & M \\ M^* & R}\bbm{x_n\\u_n}
	,\bbm{x_n\\u_n}\right\rangle \\
	&\qquad\quad+ \langle V (u_{n}-u_{n-1}),u_{n}-u_{n-1} \rangle\bigg),
\end{aligned}
\end{equation}
for some appropriate weights $Q,M,R,V$ and some arbitrarily fixed $u_{-1}\in U$, say $u_{-1}=0$, 
while satisfying some affine constraints on the input $u$ and the state $x$, which we denote 
by $\sbm{x_n\\u_{n-1}\\u_n}\in\cW$ for $n=0,1,\ldots$. In the optimal control problem 
\eqref{eq:discrPlant}--\eqref{eq:infhoriz}, the initial state $x_0$ is given, and the optimization 
variable is the control signal $\{u_n\}_{n=0}^\infty$. The optimization problem 
\eqref{eq:J}--\eqref{eq:Jconstr} below should be interpreted analogously.

One may try to solve the above optimal control problem approximately by approximating the cost \eqref{eq:infhoriz} by a finite sum. Choosing a \emph{horizon} $N$ and denoting $\bu':=(u'_{k})_{k=0}^{N-1}$, the control step $u_n$ at time $n$ is chosen as the first element $u_{*,0}$ of the minimizer
$\bu_*=(u_{*,k})_{k=0}^{N-1}$ of the cost functional 
\begin{equation}\label{eq:J}
\begin{aligned}
	 J(\bu',\sbm{x_n\\u_{n-1}})&:=\langle Px'_{N},x'_{N} \rangle + 
		\langle V_{N} u'_{N-1},u'_{N-1}\rangle \\
		&\quad +
	\sum_{k=0}^{N-1} \bigg(\left\langle\bbm{Q_k & M_k \\ M_k^* & R_k}\bbm{x_k'\\u_k'}
	,\bbm{x_k'\\u_k'}\right\rangle\\
	&\quad+\langle V_k (u'_{k}-u'_{k-1}),u'_{k}-u'_{k-1} \rangle\bigg),
\end{aligned}
\end{equation}
where the terminal penalty $P=P^*$ and the weights $Q_k=Q_k^*\in\cL(X)$, $R_k=R_k^*$, 
$V_k=V_k^* \in\cL(U)$ 
and $M_k \in\cL(U,X)$ satisfy $P,V_N,V_k, \sbm{Q_k & M_k \\ M_k^* & R_k}\geq0$ for all 
$k=0,\ldots,N-1$. The minimization of the cost functional \eqref{eq:J} is subject to 
constraints 
of the form
\begin{equation}\label{eq:Jconstr}
\begin{aligned}
	& \left\{\begin{aligned}
	x'_{k+1}&=A'x'_k+B'u'_k,\\
	x'_0&=x_n, \\
	u'_{-1} &= u_{n-1},
	\end{aligned}\right.\quad \bbm{x'_{N}\\u'_{N-1}}\in\cT,\\
	&\bbm{x'_k\\u'_k\\u'_{k-1}}\in\cW_k, \quad 0\leq k\leq N-1,
\end{aligned}
\end{equation}
where $\mathcal{W}_k$ and $\mathcal{T}$ represent some stage and terminal constraints 
introduced in \eqref{eq:onestepconstr}--\eqref{eq:onestepconstr2}, 
respectively, and ideally the \emph{prediction model}  is 
perfect, i.e., $(A',B')=(A,B)$.  The obtained control 
step $u_n$ is applied to the plant and at the next time step, a new optimization is carried out. 
We \emph{assume} that the MPC procedure is stabilizing and feasible, and refer to \cite[Sect. 
5--7]{GrPaBook} for an introduction to these topics.

By the above procedure, MPC solves a constrained quadratic optimization problem at every time step $n$ in order to compute the control action $u_n$. \emph{The point of explicit MPC} is to express the optimal input sequence $\bu'$ of \eqref{eq:J}--\eqref{eq:Jconstr} subject to the constraints \emph{explicitly} as a piecewise affine function of the parameter $\theta_n:=\sbm{x_n\\u_{n-1}}$, rather than as the implicit minimizer of the constrained optimization problem described above. 

In the sequel, we assume that the stage constraints in \eqref{eq:Jconstr} can be written in the affine form, for all $k=0,1,\ldots,N-1$, 
\begin{equation}\label{eq:onestepconstr}
\begin{aligned}
	& \bbm{x'_k\\u'_{k-1}\\u'_k}\in\cW_k\quad\iff \\
	& d_k-\cE_k x'_{k}-\cF_k u'_{k-1}-E_ku'_{k}\in [0,\infty)^{p_k},
\end{aligned}
\end{equation}
and that the terminal constraint can be written as
\begin{equation}\label{eq:onestepconstr2}
	\bbm{x'_N\\u'_{N-1}}\in \cT\quad\iff\quad \widehat d-\widehat Ex'_N-\widehat Fu'_{N-1}\in[0,\infty)^{\hat p},
\end{equation}
with $d_k\in[0,\infty)^{p_k}$, $\widehat d\in[0,\infty)^{\hat p}$, $\cE_k\in\blo(X;\R^{p_k})$, 
$\cF_k,E_k\in\blo(U;\R^{p_k})$, $\widehat F\in\blo(U;\R^{\hat p})$ and $\widehat 
E\in\blo(X;\R^{\hat p})$ all given as part of the control problem, where $\hat p$ and $p_k$ for 
all $k=0\ldots,N-1$ are nonnegative integers. Unlike our work here, in the literature, it is common 
to 
assume that the constraints on the state are decoupled from the constraints on the control. 
However, this is restrictive in practice, e.g., in the situation where one already has a low-level 
controller in closed loop with the plant that one wants to control by MPC, and one needs to avoid 
saturating the input of the low-level controller.

\begin{remark}
If there is no penalty on the rate of change of control, i.e., $V_k=0$ for all $k\in\mathbb{N}_0$ in 
\eqref{eq:infhoriz}, then one can take $\cF_k=0$ in \eqref{eq:onestepconstr} and $\widehat F=0$ in \eqref{eq:onestepconstr2}. Moreover, there may then be no need to keep track of $u'_{k-1}$, so that at time step $n$, the constraints reduce to $\sbm{x_k'\\u_k'}\in\cW_k$ for $k=0,1,\ldots, N-1$, $x_N\in\cT$, and the parameter reduces to $\theta_n=x_n$. In this case there is no need to specify a $u_{-1}$ when initializing the system at time $n=0$.
\end{remark}

\section{Writing the optimization step as a parametric quadratic program (pQP)}\label{sec:pQP}

Iterating the dynamics $x'_{k+1}=Ax'_k+Bu'_k$ of the prediction model and using $x'_0=x_n$, we get the solution formula
$$
	x'_{k}=A^{k}x_n+\sum_{j=0}^{k-1} A^jBu'_{k-1-j}, \quad k\geq0,
$$
and using this, we can derive bounded operators $\widetilde A$ and $\widetilde B$ with
$$
	\bx':=(x'_{k})_{k=1}^{N} = \widetilde A x_n+\widetilde B\bu'.
$$
Note the difference in the indexing: in the vector $\bx'$, the predicted states are at time steps $1\ldots N$ and the predicted control moves in $\bu'$ are at time steps $0\ldots N-1$.

With $\widetilde Q_P:=Q_1\oplus \cdots \oplus Q_{N-1}\oplus P\geq0$, $\widetilde 
R:=R_0\oplus\cdots\oplus 
R_{N-1}\geq0$, $\widetilde M_0:=\bbm{M_0&0&\ldots&0}$,
\begin{align*}
\widetilde M & := \begin{bmatrix}
0  & M_1 & 0 &  \ldots & 0\\
0 & 0 & M_2 & \ddots & 0 \\
\vdots & \vdots & \ddots & \ddots & \vdots \\
0 & 0 & 0 & \ddots & M_{N-1}  \\
0 & 0 & 0 & \ldots & 0 \end{bmatrix},
\qquad
\widetilde V_0 := \begin{bmatrix}
	-V_0 \\ 0 \\ \vdots \\ 0
\end{bmatrix}, \\
\widetilde V & := \bbm{
V_0+V_1  & -V_1 &0& \ldots &  0\\
-V_1 & V_1+V_2 & -V_2&\ldots & 0 \\
\vdots & \ddots & \ddots & \ddots&\vdots \\
0 & \ldots & \ddots & \ddots & -V_{N-1} \\
0 & \ldots & 0 & -V_{N-1} & V_{N-1}+V_N},
\end{align*}
and denoting $\theta_n:=\sbm{x_n\\u_{n-1}}$, the cost functional \eqref{eq:J} can be written as
\begin{equation}\label{eq:Jequal}
\begin{aligned}
	J(\bu',\theta_n) &= 
	\left\langle \bbm{Q_0+\widetilde A^*\widetilde Q_P\widetilde 
	A&0\\0&V_0}\theta_n,\theta_n\right\rangle\\
	&\qquad
	+\langle H\bu',\bu'\rangle+2\langle \bu',F\theta_n\rangle,
\end{aligned}
\end{equation}
\begin{equation}\label{eq:QPobj}
\begin{aligned}
	H &:= \widetilde B^*\widetilde Q_P\widetilde B+\widetilde R + \widetilde V + \widetilde B^*\widetilde M+\widetilde M^*\widetilde B
	\qquad\text{and}\\
	F &:= \bbm{\widetilde B^*\widetilde Q_P\widetilde A + \widetilde M^*\widetilde A+\widetilde M_0^*&\widetilde V_0}.
\end{aligned}
\end{equation}
Since we want to minimize $J(\bu',\theta_n)$ by varying $\bu'$, for a fixed, given 
$\theta_n=\sbm{x_n\\u_{n-1}}$, the first term in \eqref{eq:Jequal} does not influence the location of the minimum, and 
it can be omitted in the optimization. 

The particular case when $H$ is \emph{coercive} is important, i.e., when $H^*=H$ and
\begin{equation} \label{eq:Hgeqeps}
	\langle H\bu',\bu'\rangle \geq \varepsilon\|\bu'\|^2,
	\quad \bu' \in U^N,
\end{equation}
for some $\varepsilon > 0$ independent of $\bu'$. This is guaranteed, 
e.g., if $\widetilde V$ is coercive (see Lemma \ref{lem:tVpos} below); indeed observe that 
$$
	H-\widetilde V:=\widetilde B^*\widetilde Q_P\widetilde B+\widetilde R + \widetilde B^*\widetilde M+\widetilde 
M^*\widetilde B\geq0,
$$
since $\langle(H-\widetilde V)\bu',\bu'\rangle$ equals $J(\bu',0)$ in \eqref{eq:Jequal}, or 
equivalently in \eqref{eq:J}, with $V_k=0$ for all $k$, and we assumed 
$\sbm{Q_k&M_k\\M_k^*&R_k},P\geq0$. Moreover, $H$ inherits coercivity from $\widetilde 
R$ in case $M_k=0$ for all $k=1,\ldots,N-1$.

\begin{lemma}\label{lem:tVpos}
The operator $\widetilde V$ is positive semidefinite. If $V_k$ are coercive for all (possibly apart from one) $k \in 
\{0,1,\ldots,N\}$, then $\widetilde V$ is coercive.
\end{lemma}
\begin{proof}
For an arbitrary $\bu' \in U^N$, we have 
\begin{align*}
	\langle \widetilde V\bu',\bu' \rangle & = \langle V_0u_0',u_0' \rangle + \langle 
	V_Nu_{N-1}',u_{N-1}' \rangle\\
	&\qquad+ \sum_{k=1}^{N-1} \langle V_k(u_k'-u_{k-1}'),u_k'-u_{k-1}' 
	\rangle .
\end{align*}
The right-hand side is clearly non-negative as $\langle V_ku,u\rangle \geq 0$ for all $u\in U$ and $k \in \{0,1,\ldots,N\}$, see the paragraph after \eqref{eq:Jconstr}, and this implies that $\widetilde V$ is 
positive semidefinite. 

Let us now assume that every $V_k$ is coercive, possibly apart from one, i.e., that there exist $\varepsilon_k \geq 0$, all strictly positive, apart from at most one, such that $\langle V_ku,u\rangle  \geq \varepsilon_k\|u\|^2$  independently of $u\in U$ for all $k \in \{0,1,\ldots,N\}$. Using this and the reverse triangle inequality, we get 
\begin{equation} \label{eq:Vp1}
\begin{aligned}
	\langle \widetilde V\bu',\bu' \rangle &\geq \varepsilon_0\|u_0'\|^2 + \varepsilon_N\|u'_{N-1}\|^2 \\
	&\qquad+ 
\sum_{k=1}^{N-1}\varepsilon_k\,(\|u'_k\| - \|u'_{k-1}\|)^2 = \langle \widetilde \varepsilon \,\bu,\bu
\rangle_{{\R^N}},
\end{aligned}
\end{equation}
where we denote $\bu := (\|u'_k\|)_{k=0}^{N-1}{\in\R^N}$ and $\widetilde\varepsilon$ is defined as 
$\widetilde V$ {with $V_k$ replaced by $\varepsilon_k$}; hence $\widetilde\varepsilon$ is an 
$N\times N$ matrix, rather than an operator on $U^N$.

The right-hand side in \eqref{eq:Vp1} is non-negative, and it can only be zero if $\|u_k\|  = 0$ for all 
$k\in\{0,1,\ldots,N-1\}$, i.e., $\bu = 0$, and this implies that $\widetilde\varepsilon > 0$. 
Consequently, we have $\langle \widetilde\varepsilon\bu,\bu\rangle \geq 
\lambda_{\min}(\widetilde\varepsilon)\|\bu\|^2$, where $\lambda_{\min}(\widetilde\varepsilon) > 
0$ denotes the smallest eigenvalue of $\widetilde\varepsilon$. Finally, noting that $\|\bu\|^2 = 
\|\bu'\|^2$, we have that $\widetilde V$ is coercive:
$$
	\langle \widetilde V\bu',\bu'\rangle\geq
	\langle \widetilde\varepsilon\bu,\bu\rangle\geq
	\lambda_{\min}(\widetilde\varepsilon)\,\|\bu'\|^2.
$$
\end{proof}

Consequently, one can set $V_0 = 0$ to avoid having 
to specify $u_{-1}$ when initializing the system at time $n=0$, in case $u'_{k-1}$ does not appear in the stage constraints \eqref{eq:onestepconstr}. 

Completing the square in \eqref{eq:Jequal}, one gets that
\begin{equation}\label{eq:QPobjz}
	f(\bz):=\frac 12 \langle H\bz,\bz\rangle
\end{equation}
satisfies the following for all $\theta_n=\sbm{x_n\\u_{n-1}}$:
$$
	f(\bz)=\frac12 J(\bu',\theta_n)+g(\theta_n),\quad\text{as long as}\quad	
	\bz=\bu'+H^{-1}F\theta_n;
$$
moreover the inverse $H^{-1}$ is bounded with norm at most $1/\varepsilon$ because of \eqref{eq:Hgeqeps}. Hence, we will 
focus on minimizing $f(\bz)$.

Letting $\widetilde p<\infty$ denote the total number of constraints in  \eqref{eq:onestepconstr}--\eqref{eq:onestepconstr2}, one can in a similar way to the above derive a vector $W\in\R^{\widetilde p}$ and bounded linear operators $S\in\cL(X\times U,\R^{\widetilde p})$ (or $S\in\cL(X,\R^{\widetilde p})$) and $G\in\cL(U^N,\R^{\widetilde p})$, such that all these constraints can be 
written compactly~as 
\begin{equation}\label{eq:WSGintro}
	W+S\theta_n-G\bz\geq0,
\end{equation}
where the inequality is understood componentwise in $\R^{\widetilde p}$.

\section{Explicit solution of the pQP}\label{sec:expsol}

We will now solve explicitly the minimization of $f$ in \eqref{eq:QPobjz} subject to the constraints 
\eqref{eq:WSGintro}, i.e., the following pQP, which is to be solved by the MPC controller at 
every time step:
\begin{equation}\label{eq:mpQP}
\operatorname{argmin}_{\,\bz\in U^N} \frac 12 \langle H\bz,\bz\rangle
\quad\text{s.t.}\quad
W+S\theta_n-G\bz\geq0,
\end{equation}
where $W\in\R^{\widetilde p}$ and $H$, $S$, $G$ are all bounded operators. 

We start with a simple result on existence and uniqueness, which is contained in \cite[Prop. 
2.3.3, Thm 3.3.4]{ABMBook}. 

\begin{lemma}\label{lem:uniqueness}
Assume that the optimization problem \eqref{eq:mpQP} is \emph{feasible}, i.e., that the \emph{feasible set} $\cS(\theta_n)$, which consists of all $\bz$ satisfying the constraint in 
\eqref{eq:mpQP}, is nonempty. Moreover, assume that the 
operators $G,H$ are bounded and that $H$ is \emph{coercive}. Then 
\eqref{eq:mpQP} has a unique solution $\bz_ *$.
\end{lemma}

We next work towards a Hilbert-space extension of the KKT optimality conditions used in the explicit MPC literature. The first step is taken in the following lemma, whose second part is also contained in \cite[Thm 9.6.1]{ABMBook}.

\begin{lemma}\label{lem:KT}
Let $Z$ be a \emph{real} Hilbert space and let $f$ and $f_k$, $1\leq k\leq \widetilde p$, be a \emph{finite number} of convex functionals on $Z$. Moreover assume that $f$ and all $f_k$ are G\^ateaux differentiable on $Z$. Then statement 2) implies statement 1) below: 
\begin{enumerate}
\item The point $z_ *\in Z$ minimizes $f(z)$ subject to the constraints $f_k(z)\geq0$ for all $k=1,\ldots,\widetilde p$.

\item The point $z_*\in Z$ satisfies $f_k(z_*)\geq 0$ for all $k=1,\ldots,\widetilde p$, and there are $\lambda_k\geq0$, $k=1,\ldots,\widetilde p$, such that:
\begin{enumerate}
\item $\lambda_kf_k(z_ *)=0$ for all $k=1,\ldots,\widetilde p$, and 
\item the G\^ateaux derivatives of $f$ and $f_k$ satisfy
$$
	f'(z_ *)=\sum_{k=1}^{\widetilde p} \lambda_k f_k'(z_ *).
$$
\end{enumerate}
\end{enumerate}
Assume that the constraints satisfy the following \emph{Slater condition}: There exists a $z\in Z$, such that $f_k(z)>0$ for all $k=1,\ldots,\widetilde p$. Then the implication from 1) to 2) also holds.
\end{lemma}

\begin{proof}
First note that all Hilbert spaces are locally convex. The implication from 2) to 1) is then contained in statement (4) in \cite[Cor.\ 47.14]{Zei85}, after a change of signs on the constraints, whereas the converse implication follows from statements (1) and (3) in \cite[Thm 47.E]{Zei85}, assuming the Slater condition. 
\end{proof}

In case the optimization is carried out over a finite-dimensional space and the constraints are affine, it is well known that the Slater condition is not needed. It turns out that this is the case in infinite dimensions too, as long as the number of constraints is finite; this will be established as item 2) in Thm \ref{thm:affine} below.

For simplicity we introduce the notation $\dP:=\set{1,2,\ldots,\widetilde p}$. 
A vector $\lambda=(\lambda_k)_{k=1}^{\widetilde p}$ with the properties in Lemma \ref{lem:KT}.2) is referred to as a 
\emph{Lagrange multiplier associated to} $z_*$. Note that we do not claim that associated 
Lagrange multipliers are uniquely determined by $z_*$. The property $\lambda_kf_k(z_ *)=0$ for 
all $k$ is referred to as \emph{complementarity}. The set of constraints which are active at the optimizer $z_*$, i.e.,
$$
	\dA:=\set{k\in\dP \mid f_k(z_ *)=0},
$$ 
is referred to as the \emph{optimal active set}, and its complement is $\dA^{\!c}:=\dP\setminus\dA$. Denote 
$$
	\cI_k:=\bbm{0&\ldots&0&1&0&\ldots&0}\in\R^{\widetilde p}, 
$$
with the one in position $k\in\dP$. For any index set $\dA$ and operator $G$ mapping into 
$\R^{\widetilde p}$, we use $\#\dA$ to denote the number of elements in $\dA$, and $G^\dA$ 
means $G$ projected onto the components in $\R^{\widetilde p}$ which are indexed in $\dA$, so 
that $G^\dA=\Gamma G$, where $\Gamma=(\cI_k)_{k\in\dA}\in\R^{(\#\dA)\times \widetilde p}$, 
with $k$ in increasing order.

We next apply Lemma \ref{lem:KT} to \eqref{eq:mpQP}, for a fixed 
parameter value $\theta_n\in X\times U$, in order to get a Hilbert-space analogue of \cite[Thm 2]{BMDP02}. The resulting theorem is stronger than most similar results presently in the 
explicit MPC literature, because we obtain that the minimizer is an affine function of the 
parameter $\theta_n$, without assuming that the \emph{linear independence constraint 
qualification (LICQ)} holds at $\bz_*(\theta_n)$, i.e., we do not need that $G^\dA$ is surjective. In fact, if $H$ is invertible, then
\begin{align*}
	\range{G^\dA H^{-1}(G^\dA)^*} & = 
	\Ker{G^\dA H^{-1}(G^\dA)^*}^\perp \\ & =
	\Ker{(G^\dA)^*}^\perp=
	\range{G^\dA},
\end{align*}
so that the LICQ condition holds if and only if the matrix $G^\dA H^{-1}(G^\dA)^*$ is invertible. In the theorem, we more generally replace this inverse by the pseudoinverse $(G^\dA H^{-1}(G^\dA)^*)^{[-1]}$, where we denote
$$
	P^{[-1]}:=P\big|_{\Ker P^\perp}^{-1},\quad
	\dom{P^{[-1]}}:=\range P.
$$

\begin{theorem}\label{thm:affine} 
Let $\emptyset\neq\dA\subset\dP$ be a candidate active set, and let $\lambda:=(\lambda_k)_{k=1}^{\widetilde p}$ be a column vector of putative Lagrange multipliers. In the notation of \eqref{eq:mpQP}, the following are true:

\begin{enumerate}
\item Assume that $\dA$, $\theta_n\in X\times U$ (or $\theta_n\in X$), $\bz_ *\in U^{N}$ and 
$\lambda^\dA\in\R^{\#\dA}$ are such that 
\begin{equation}\label{eq:KKTsuff}
	\begin{split}
W^\dA+S^\dA \theta_n & =G^\dA\bz_ *,\\
W^{\dA^c}+S^{\dA^c}\theta_n&\geq G^{\dA^c}\bz_ *, \\
\lambda^\dA  \geq0,\quad 
H\bz_ *&=-(G^\dA)^*\lambda^\dA,
\end{split}
\end{equation}
where $\dA^c:=\dP\setminus\dA$. Then $\bz_*$ is the global minimizer of \eqref{eq:mpQP}, and the active set at $\bz_*$ contains $\dA$.

\item Conversely, if $\bz_ *$ minimizes \eqref{eq:mpQP} and $\dA$ is the set of all constraints active at $\bz_ *$, then \eqref{eq:KKTsuff} holds together with $W^{\dA^c}+S^{\dA^c}\theta_n>G^{\dA^c}\bz_ *$ and $\lambda^{\dA^c}=0$.

\item If $H$ is coercive and \eqref{eq:KKTsuff} holds, then the unique minimizer of \eqref{eq:mpQP} is 
\begin{equation}\label{eq:zsol}
\bz_ *=H^{-1}(G^\dA)^*\big(G^\dA H^{-1}(G^\dA)^*\big)^{[-1]}(W^\dA+S^\dA \theta_n)
\end{equation}
and there is some 
$k\in\Ker{(G^\dA)^*}$, such that
\begin{equation}\label{eq:lambdasol}
\lambda^\dA=-\big(G^\dA H^{-1}(G^\dA)^*\big)^{[-1]}(W^\dA+S^\dA \theta_n)\oplus k,
\end{equation}
\end{enumerate}
where $h\oplus k$ denotes $h+k$ with $h\perp k$.
\end{theorem}

We say that the point $\bz_*\in U^N$ is \emph{admissible} if $W+S\theta_n\geq G\bz_ *$, i.e., $\bz_*$ lies in the feasible set $\cS(\theta_n)$. We call a candidate active set $\dA$, for which there exists some $\lambda$ that satisfies \eqref{eq:KKTsuff}, a \emph{sufficient active set}, since it may be 
strictly smaller than the set of \emph{all} constraints active at the optimum $\mathbf z_*$, but it is nevertheless sufficient for guaranteeing optimality and computing the minimizer. 

If LICQ holds at $\theta_n$ then $k=0$ and the pseudoinverse equals the standard inverse in \eqref{eq:lambdasol} and \eqref{eq:zsol}. Theorem \ref{thm:affine} 
says nothing for $\dA=\emptyset$, but it is clear that the optimizer is $\mathbf z_*=0$ in this case, provided that it is admissible, and $\lambda=0$ works as associated Lagrange 
multiplier.

\begin{proof}[Proof of Theorem \ref{thm:affine}]
Clearly, the assumptions in item 1 (in item 2) imply admissibility of $\bz_*$, and hence feasibility of \eqref{eq:mpQP}. Now fix a parameter $\theta_n\in X\times U$, such that $\cS(\theta_n)\neq\emptyset$.

The functionals $f(\bz):=\frac12\Ipdp{H\bz}{\bz}_{U^N}$ and 
$$
	f_k(\bz,\theta_n):=\cI_k\,(W+S\theta_n-G\bz), \qquad1\leq k\leq \widetilde p,
$$
are all convex: For all $\bz,\bx\in U^{N}$ and $h\in[0,1]$, using \eqref{eq:Hgeqeps} at the end, possibly with $\varepsilon=0$ unless $H$ is coercive,
$$
\begin{aligned}
	(1-h)f(\bx)+hf(\bz)-f\big((1-h)\bx+h\bz\big) &= \\
	h\frac12\langle H\bz,\bz\rangle+(1-h)\frac12\langle H\bx,\bx\rangle&\\
		-\frac12\langle H(h\bz+(1-h)\bx),h\bz+(1-h)\bx\rangle &=\\
	h(1-h)\frac12\langle H(\bz-\bx),\bz-\bx \rangle&\\
	\geq \frac{\varepsilon\, h\, (1-h)}{2}\cdot\|\bz-\bx\|^2&,
\end{aligned}
$$
and $f_k(\cdot,\theta_n)$ is also convex:
$$	
	h\,f_k(\bz,\theta_n)+(1-h)\,f_k(\bx,\theta_n)=f_k(h\bz+(1-h)\bx,\theta_n).
$$
These are Fr\'echet (hence G\^ateaux) differentiable on $U^{N}$:
$$
	\big|f(\bz+\bx)-f(\bz)-\langle H\bz,\bx \rangle\big| = 
	\frac{|\langle H\bx,\bx\rangle|}{2}\leq
	\frac{\|H\|}{2}\cdot\|\bx\|^2,
$$
so that $f'(\bz)=H\bz$, $\bz\in U^{N}$, and
$$
	f_k(\bz+\bx,\theta_n)-f_k(\bz,\theta_n) +\langle G^*\cI_k^*,\bx\rangle=0,
$$
\begin{equation}\label{eq:fkder}
	\text{so that}\qquad f_k'(\bz, \theta_n)=- G^*\cI_k^*, \quad\bz\in U^{N};
\end{equation}
see for instance \cite[\S40.1]{Zei85}.

According to Lemma \ref{lem:KT}, $\bz_ *$ solves \eqref{eq:mpQP} if $\bz_*\in\cS(\theta_n)$ and there exist Lagrange multipliers $\lambda_k\geq0$ with $\lambda_kf_k(\bz_ 
*)=0$ for all $1\leq k\leq \widetilde p$, such that
\begin{equation}\label{eq:KTvari}
	f'(\bz_ *)=\sum_{k=1}^{\widetilde p} \lambda_kf_k'(\bz_ *).
\end{equation}
By substituting the Fr\'echet derivatives calculated above, we get that \eqref{eq:KTvari} holds if and only if $H\bz_ *+G^*\lambda=0$. If $\lambda_k=0$ for all $k\not\in\dA$, then  $H\bz_ *+G^*\lambda=0$ collapses into $H\bz_ *+(G^\dA)^*\lambda^\dA=0$. All constraints in $\dA$ are active at the 
optimum $\bz_ *$ if and only if $G^\dA\bz_ *=W^\dA+S^\dA \theta_n$. Hence, \eqref{eq:KTvari} holds, $\lambda^{\dA^c}=0$ and all constraints in $\dA$ are active if and only if
\begin{equation}\label{eq:KKTsys}
	\bbm{H&(G^\dA)^*\\G^\dA&0}\bbm{\bz_ *\\\lambda^\dA} =
	\bbm{0\\W^\dA+S^\dA \theta_n}
\end{equation}
and $\lambda^{\dA^c}=0$. Assuming additionally that $\lambda^\dA\in[0,\infty)^{\#\dA}$ and $W^{\dA^c}+S^{\dA^c}\theta_n\geq G^{\dA^c}\bz_ *$, we get from Lemma \ref{lem:KT} that $\bz_*$ minimizes \eqref{eq:mpQP}. Hence whenever \eqref{eq:KKTsuff} holds, we obtain that $\bz_*$ minimizes \eqref{eq:mpQP} by \emph{defining} $\lambda^{\dA^c}:=0$. Noting that some constraints $k\not\in\dA$ may also be active at $\bz_*$, we get that the active set at $\bz_*$ is $\dB_{\theta_n}\supset\dA$. The proof of item 1 is complete.

In order to prove item 2, we will apply \cite[Thm 2.4]{TaTr94}. For this, we equip all finite sets $\cF$ with the discrete topology, so that the $\sigma$-algebra of $\cF$ equals the power set of $\cF$. Thus all finite sets induce measure spaces where all subsets are measurable, and the standing assumptions on p. 5 of \cite{TaTr94} are satisfied in a rather trivial manner. By the first paragraph of \cite[\S3]{TaTr94}, \cite[Thm 2.4]{TaTr94} is applicable, since \eqref{eq:mpQP} has only a finite number of constraints, i.e., $\#\dP=\widetilde p<\infty$.

In order to make use of \cite[Thm 2.4]{TaTr94}, we next need to verify the constraint qualification in \cite[Def. 2.2]{TaTr94}: For every $\bz\in\cS(\theta_n)$, and for every nonzero $\bx\in U^N$ such that $\ipdp{\bx}{f_k'(\bz,\theta_n)}\geq0$ for all $k\in\dA_\bz$, there exist $\tau>0$ and a continuous arc $C:[0,\tau)\to U^N$, such that
$$
	C(0)=\bz,\quad C'(0)=\bx,\quad\text{and}\quad f_k\big(C(t),\theta_n\big)\geq0
$$
for all $t\in[0,\tau)$ and all $k\in\dP$; here $\dA_\bz$ denotes the set of constraints that are active at $\bz$:
$$
	\dA_\bz:=\set{k\in\dP\mid f_k(\bz,\theta_n)=0}.
$$
We fix some arbitrary pair $(\bz,\bx)$ with the above properties, and we will construct a feasible arc $C$ with the required initial conditions. Since our constraints are affine, we choose $C(t):=\bz+t\bx$, $t\in[0,\tau)$, as suggested on p.\ 11 of \cite{TaTr94}. Then
\begin{equation}\label{eq:arccond}
\begin{aligned}
	f_k\big(C(t),\theta_n\big) &= \cI_k\big(W+S\theta_n-G(\bz+t\bx)\big) \\
	&= f_k(\bz,\theta_n)+t\,\ipdp{\bx}{f'_k(\bz,\theta_n)}\geq0
\end{aligned}
\end{equation}
for all $t\geq0$ and $k\in\dB$ with
$$
	\dB:=\set{k\in\dP\mid \ipdp{\bx}{f'_k(\bz,\theta_n)}\geq0}. 
$$
Now let 
$$
	a:=\min_{k\in\dP\setminus\dB} f_k(\bz,\theta_n)>0
$$ 
(because $\dB\supset\dA_\bz$ by the choice of $\bz$ and $\bx$) 
and 
$$
	b:=\max_{k\in\dP\setminus\dB} -\ipdp{\bx}{f'_k(\bz,\theta_n)}>0
$$ 
(because of the definition of $\dB$). Then \eqref{eq:arccond} is satisfied also for all $k\in\dP\setminus\dB$ and for all $t\in[0,\tau)$, with $\tau:=a/b>0$. Thus the constraint qualification holds.

Let $\bz_*$ be a minimizer of \eqref{eq:mpQP} with the corresponding active set $\dA$. By \cite[Thm 2.4]{TaTr94}, the equivalent conditions in \cite[Prop.\ 2.2]{TaTr94} hold, i.e.,  there exists a finite measure $u^*$ on $\dP$, such that,
\begin{enumerate}
\item in a sense made precise in \cite{TaTr94},
\begin{equation}\label{eq:uint}
f'(\bz_*)=\int_\dP f'_k(\bz_*,\theta_n)\, du^*,
\end{equation}
\item $u^*(\dP')=0$ for all $\dP'\subset\dA^c$ and
\item $u^*(\dP')\geq0$ for all $\dP'\subset\dP$.
\end{enumerate}
Next define $\lambda_k:=u^*(\set{k})$ for $k\in\dP$. By properties 2 and 3 of $u^*$, $\lambda_k=0$ for all $k\in\dA^c$ and $\lambda_k\geq0$ for all $k\in\dA$. Since $\dP$ is finite, the integral \eqref{eq:uint} simplifies, using \eqref{eq:fkder} and the corresponding expression for $f'$, to
$$
	H\bz_*=\sum_{k\in\dP}
	-G^*\cI_k^*u^*(\set k)=-G^*\lambda
	=-(G^\dA)^*\lambda^\dA,
$$
because $\lambda^{\dA^c}=0$. This completes the proof of item 2.

Finally, for item 3, note that by \eqref{eq:KKTsuff}, the problem is feasible, and then Lemma \ref{lem:uniqueness} gives (the existence and) the uniqueness of the minimizer $\bz_*$. The conditions \eqref{eq:KKTsuff} in particular imply \eqref{eq:KKTsys}, and solving the latter for $\lambda^\dA$, we get \eqref{eq:lambdasol}, and then \eqref{eq:zsol} follows easily. 
\end{proof}

Due to Theorem \ref{thm:affine}, the dual active set method QPKWIK algorithm, the basis of the
\texttt{mpcActiveSetSolver} algorithm in MATLAB,  is applicable to solving the strictly convex quadratic programming problem 
\eqref{eq:mpQP}. This method is too involved to be 
reproduced here, but some of its nice properties are:

\begin{enumerate}
\item It is easy to find a dually feasible initial point for the iteration, namely the unconstrained 
optimum.
\item LICQ is in general not maintained throughout optimization, but one has more leeway in avoiding LICQ-related problems than is the case with primal active set methods.
\item It solves \eqref{eq:mpQP} or determine that no solution exists in a finite 
number of steps, since $\widetilde p<\infty$, and since each candidate active set is visited at 
most once.
\end{enumerate}
We close this section by observing that the mentioned algorithm in general operates only with 
\emph{sufficient} active sets rather than \emph{full optimal active set}.

\section{Example: Explicit MPC of a Timoshenko beam}\label{sec:example}

\subsection{Port-Hamiltonian formulation of Timoshenko beam}

As in \cite[Ex. 2.19]{VilPhd} (see also \cite[Ex. 7.1.4]{JacZwaBook}), we consider the 
Timoshenko 
beam model on the spatial interval $\xi\in[0,1]$. Let $w(\xi,t)$ and $\phi(\xi,t)$ denote the 
transverse displacement and the rotation angle of the beam, respectively, and let 
$\rho,I_{\rho},EI,K \in C^1(0,1;\mathbb{R}_+)$ denote mass per unit 
length, the rotary moment of inertia of the cross section, the product
of Young's modulus of elasticity and the moment of inertia of the
cross section, and the shear modulus, respectively. By defining a state variable $x =  
(x_1,x_2,x_3,x_4) := (w'-\phi, \rho \dot{w}, \phi', I_\rho \dot{\phi})$ where $()'$ and $\dot{()}$ 
denote spatial and temporal derivatives, respectively, the beam model can be written as a 
port-Hamiltonian system
\begin{equation}
	\dot{x}(t) = P_1\big(\mathcal{H}x(t)\big)' +
	P_0\mathcal{H}x(t), \quad x(0) = x_0,
	\label{eq:phs}
\end{equation}
where $\mathcal{H} =
\operatorname{diag}(K, 1/\rho, EI,
1/I_{\rho})$
and 
$$
P_1 = 
\begin{bmatrix}
	0 & 1 & 0 & 0 \\ 1 & 0 & 0 & 0 \\ 0 & 0 & 0 & 1 \\ 0 & 0 & 1 & 0 
\end{bmatrix}, \qquad P_0 = 
\begin{bmatrix}
	0 & 0 & 0 & -1 \\ 0 & 0 & 0 & 0 \\ 0 & 0 & 0 & 0 \\ 1 & 0 & 0 & 0
\end{bmatrix}.
$$

We consider a cantilever beam, where we control the free end.
That is,  assuming that the end at $\xi = 0$ is clamped, the boundary
conditions and controls of \eqref{eq:phs} are given by
\begin{subequations}
\begin{align}
	\frac{1}{\rho(0)}x_2(0, t) & = 
	\frac{1}{I_{\rho}(0)}x_4(0,t) \equiv 0 \\
	u_1(t) & = K(1)x_1(1,t) \\
	u_2(t) & = EI(1)x_3(1,t).
\end{align}
	\label{eq:bcs}%
\end{subequations}
By checking the conditions of  \cite[Thm 13.2.2]{JacZwaBook}, we have that the system 
\eqref{eq:phs}--\eqref{eq:bcs} is well-posed on $X = L^2(0,1;\mathbb{R}^4)$, i.e., it has a 
well-defined, unique weak solution for any inputs $u_1,u_2\in L^2_{loc}(0,\infty; \mathbb{R})$ and 
initial conditions $x_0\in X$.

\subsection{Cost functional  and conversion to discrete time}

Consider the cost functional
\begin{equation}
  \label{eq:ocp}
  \int\limits_0^{\infty} \left\langle Q'x(t),x(t)
  \right\rangle + \left\langle R'u(t), u(t) \right\rangle + \left\langle V'\dot u(t), \dot u(t) \right\rangle 
  dt,
\end{equation}
with weights $Q', R', V' \geq 0$ and with constraints 
$\int_0^1x_1(t) dt \leq \overline{x}_1$,
$\int_0^1x_4(t)dt \geq \underline{x}_4$, and 
$\underline{u} \leq u_{1,2}(t) \leq \overline{u}$
for all $t\geq 0$, where the bounds $ \overline{x}_1, \underline{x}_4, \underline{u},
\overline{u}$ will be specified in \S\ref{ex:sim}. Next, we will convert  the continuous-time problem 
\eqref{eq:ocp} into discrete-time and formulate the MPC procedure to find a control signal that 
approximately minimizes \eqref{eq:ocp}.

The system \eqref{eq:phs}--\eqref{eq:bcs} is transformed to discrete time using the Cayley 
transform \cite[(1.5)]{HaMa07}. For the boundary control system \eqref{eq:phs}--\eqref{eq:bcs}, 
the Cayley transform can be computed explicitly based on the Laplace transform of  
\eqref{eq:phs} and \cite[Rem. 10.1.5]{TuWeBook} similar to \cite[Sect. 4.1]{DuHu20}. However, 
the closed-form 
expressions of the discrete-time operators would be lengthy, so for practical computations 
and computational efficacy, we prefer computing the discrete-time operators based on 
finite-dimensional approximations of \eqref{eq:phs}--\eqref{eq:bcs}. We will specify the employed 
approximation methods in \S\ref{ex:sim}.

Converting the cost function \eqref{eq:ocp} into discrete time, the
coefficients $Q', R', V'$ should be scaled by the time discretization interval $h$ to account 
for the temporal integration (and to make the cost function independent of
$h$). However, we need to take into account that $u_k/\sqrt{h}$ approximates
$u(t)$ in the Cayley transform. Thus, the the discrete-time
cost function on some finite horizon $[0, N-1]$ is given by 
\begin{equation}
  \label{eq:docp}
  \begin{aligned}
  &\sum_{k=0}^{N-1} \Big(\left\langle Qx_k,x_k \right\rangle +
  \left\langle Ru_k,u_k \right\rangle \\
  &\qquad+ \left\langle
    V(u_k-u_{k-1}),u_k-u_{k-1} \right\rangle\Big),
\end{aligned}
\end{equation}
where $Q = hQ', R=R'$, and $V = h^{-2}V'$ (scaling by
$h^{-2}$ comes from the finite-difference approximation of
$\dot{u}(t)$). Compared to \eqref{eq:J}, we have taken $V_N=0$ as we see no reason to 
additionally penalize $u_{N-1}\neq 0$, albeit for longer prediction horizons this has very little 
impact on the solution. Moreover, even if we have not added any terminal ingredients to 
\eqref{eq:docp}, based on the simulation results of the next section, the MPC procedure is 
stabilizing. While stability of MPC is outside of our scope, we note that the turnpike property of 
\emph{unconstrained} quadratic problems of the form \eqref{eq:ocp} (with $V =0$) has been 
considered in \cite{GSS20}, which would provide one way of guaranteeing stability of MPC 
without terminal ingredients. However, these results cannot be (directly) applied to 
\emph{constrained} problems (with $V>0$), so it is unclear whether the constrained minimization 
problem of \eqref{eq:ocp} has the turnpike property or not.

\subsection{Numerical simulation} \label{ex:sim}

In the simulation, the input constraints are set to $u_{1,2} \in [-0.5, 0.5]$ and the state 
constraints are $\underline x_4 = -0.3$ and $\overline x_1 = 0.45$. The weights in the cost 
function are $Q'=100$, $R = 1$, $V' =0.1$, and the time discretization is $h=2^{-7}$. The MPC 
prediction horizon is set to $N=30$. For simplicity, the physical parameters of the beam are set to 
$1$. The initial conditions are given by $x_1 =x_4 = 0, x_2=\sin(\frac{\pi}{2}\xi)$ and $x_3 = 
\cos(\frac{\pi}{2}\xi)$, and we initialize $u_{-1}= 0$.

In order to demonstrate a somewhat realistic scenario, where the prediction model is not a 
perfect copy of the system dynamics, we employ different approximations for the plant and 
the prediction model. For the plant, we approximate the beam model by finite differences, and for 
the prediction model we employ a spectral-Galerkin approximation. In the finite-difference 
approximation of the plant we use $127$ grid points, and in the spectral-Galerkin approximation 
we use $9$ polynomial basis functions. Thus, the prediction model is lower-dimensional by a 
factor of $14$, which emulates the difference between an infinite-dimensional plant a 
finite-dimensional prediction model. Moreover, the continuous-time nature of the plant is 
emulated by simulating the plant using the \texttt{ode45} solver in MATLAB as opposed to using a 
prescribed time-discretization in the simulation.

The simulation results are displayed in Figures \ref{fig:uJ2}--\ref{fig:ss2}. First, Figure 
\ref{fig:uJ2} 
displays the optimal controls solved using the \texttt{mpcActiveSetSolver} in MATLAB and the 
corresponding optimal costs. It can be seen that the controls satisfy the input constraints and that 
the optimal costs are (exponentially) decreasing during the first part of the simulation. However, 
the evolution of the optimal costs 
experiences some ripples towards the end, even if the long-term trend is still decreasing, likely due to the imperfections in the prediction model. 
Moreover, the cost function values are already very close to zero when the ripples occur.

\begin{figure}[htpb]
	\includegraphics[width=\linewidth]{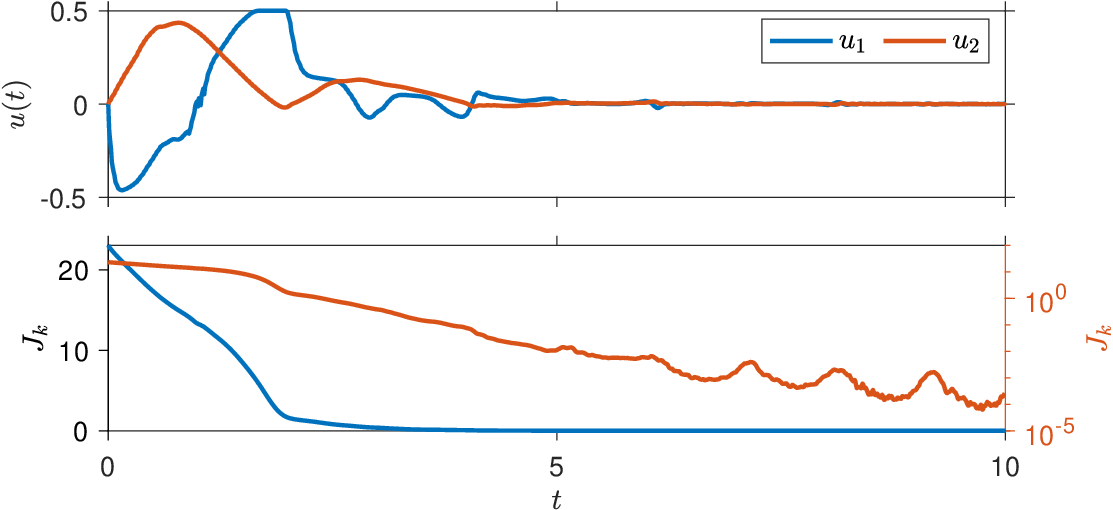}
	\caption{Controls and the optimal costs (logarithmic scale on the right).}
	\label{fig:uJ2}
\end{figure}

Figure \ref{fig:sn2} shows the mean values of the state components over the spatial interval 
along with the state constraints. While it appears that the state constraints become active at some 
point, a closer inspection of Figure \ref{fig:sn2} reveals that the constraint $\overline x_1$ in fact 
gets violated (very slightly) as the maximum value of mean$(x_1)$ is $3.78\cdot 10^{-4}$ larger 
than the upper bound $\overline x_1 = 0.45$. This violation of 
the state constraint is very mild, and we view the constraint as still being satisfied in the soft 
sense \cite[Sect. 6.3]{BMDP02}. We note that strict satisfaction of state constraints 
under 
imperfect prediction model would require robustness considerations which are outside the scope 
of this study. 
Finally, we note that $\underline x_4$ does not strictly get activated as the minimum value of 
mean$(x_4)$ is $7.61\cdot 10^{-4}$ larger than the lower bound $\underline x_4 = -0.3$.

For additional illustration, we also present the state profile of the plant component $x_3$ in 
Figure \ref{fig:ss2}, where the effect of an imperfect prediction model can also be seen. 
Somewhat similar to the cost function values in Figure \ref{fig:uJ2}, small ripples appear in the 
state profile as the value approaches zero, albeit these may also be caused by numerical noise. 
Interestingly enough, these ripples did not appear in 
Figure \ref{fig:sn2}---possibly due to the averaging effect of the mean value. Regardless, the 
state 
component seems to remain in some $\delta$-neighborhood of zero once it enters there, which 
can viewed as a version of practical stability \cite[Def. 2.15]{GrPaBook}.

\begin{figure}[!htpb]
	\includegraphics[width=\linewidth]{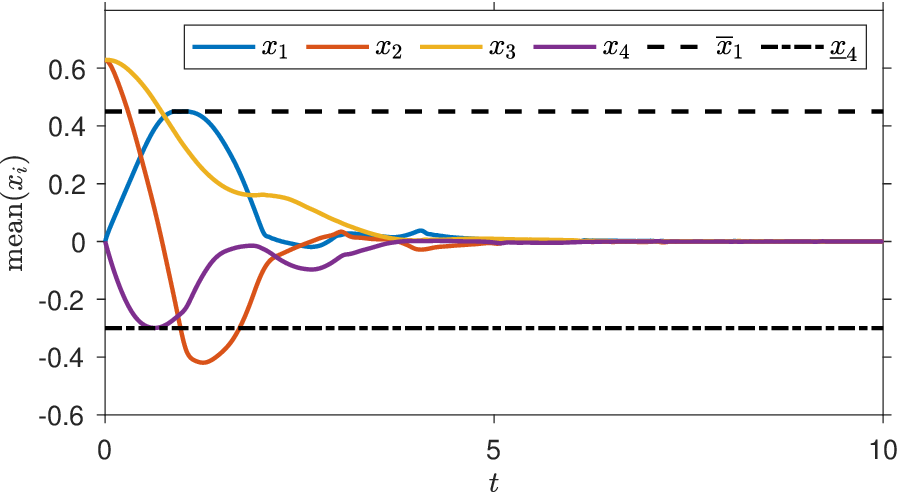}
	\caption{Means of the plant state components and the state constraints.}
	\label{fig:sn2}
\end{figure}

\begin{figure}[htpb]
	\includegraphics[width=\linewidth]{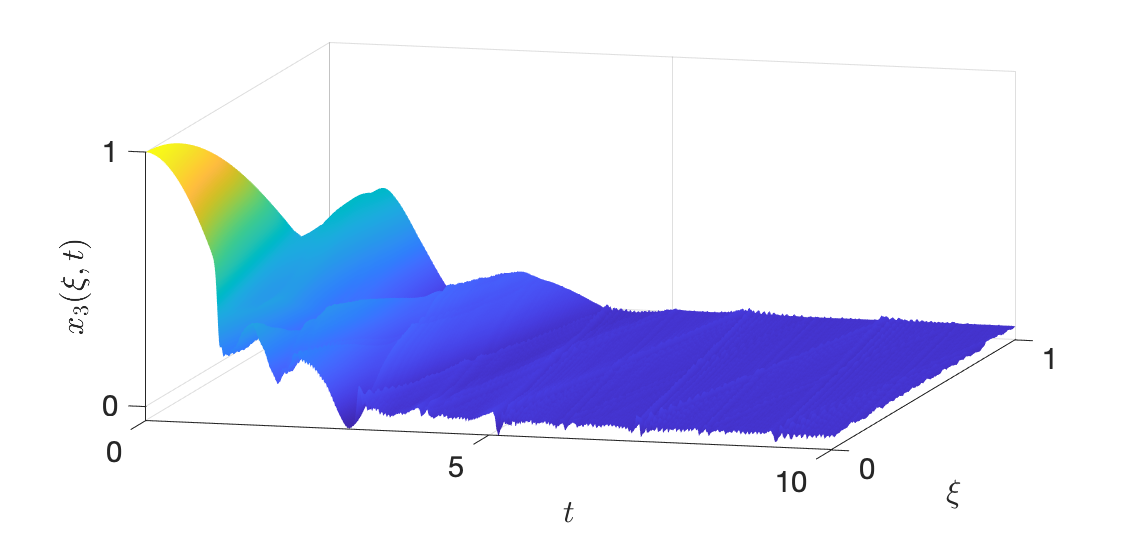}
	\caption{Profile of the plant state component $x_3 = \phi'$.}
		\label{fig:ss2}
\end{figure}

To conclude this section, we briefly comment on the choice of the prediction horizon $N$. In 
Table \ref{tab:ncomp}, we compare the computational times and optimal costs of the 
\texttt{mpcActiveSetSolver} for different horizon lengths. It can be seen that the algorithm is very 
efficient, even if the computational time increases along with the horizon length, arguably due to 
the exponentially growing number of candidate sets shown in the rightmost column of the table. 
Regardless, even the longest tested horizon $70$ is easily feasible for the ten-second simulation 
run. One explanatory factor in the rapid performance of the algorithm may be that we use the 
optimal active set of the previous control step as the initial guess for the next one, and if this set 
does not change much between steps---which often seems to be the case in this simulation---the 
algorithm finds an optimal active set already after a few iterations.

\begin{table}[htpb]
	\begin{center}
		\begin{tabular}{c|c|c|c}
			$N$ & time (s) &  $J_d$ & $2^{\widetilde p}$\\
			\hline &&& \\ [-7pt]
			10 & 0.34  & {148} & $2.88\cdot10^{17}$\\
			20 & 0.54  & {126} & $3.32\cdot10^{35}$\\
			30 & 0.84  & {122} & $2.83\cdot10^{53}$\\
			40 & 1.34 & {121} & $4.41\cdot10^{71}$\\
			50 & 2.20  & {120} & $5.06\cdot10^{89}$ \\
			60 & 3.38  & {119} & $5.87\cdot10^{107}$ \\
			70 & 5.34 & {119} & $6.77\cdot10^{125} $
		\end{tabular}
	\end{center}
	\caption{Comparison of computational times and control costs for different prediction horizon 
		lengths $N$.}
	\label{tab:ncomp}
\end{table}

Table \ref{tab:ncomp} additionally shows that the optimal cost decreases with the optimization horizon length, as expected. However, the optimal control costs $J_d$ do not noticeably change after the shortest horizon $N=10$, nor would Figures \ref{fig:uJ2}--\ref{fig:ss2}. Moreover, with longer prediction horizons, the imperfections may 
accumulate, and hence the late-horizon predictions may become unreliable. Hence, one might 
want to prioritize the early-horizon predictions with time-varying weights in the cost functional, or 
simply employ a shorter prediction horizon as we did in the simulation.

\section{Acknowledgments}

The authors gratefully acknowledge the input of the audience at the conference MTNS 2022, and Profs. Martin Mönnigmann, Mark Cannon, Eric Kerrigan \nocite{BCK16} and Tor Arne Johansen, for enlightening discussions during the preparation of this manuscript. An anonymous referee pointed out reference \cite{TaTr94}, leading to a major improvement of Thm \ref{thm:affine}.


\ifCLASSOPTIONcaptionsoff
\newpage
\fi


\begin{thebibliography}{10}

\bibitem{BMDP02}
A. Bemporad, M. Morari, V. Dua, and E.~N. Pistikopoulos,
\newblock ``The explicit linear quadratic regulator for constrained systems,''
\newblock {\em Automatica J. IFAC}, vol. 38, no. 1, pp. 3--20, 2002.

\bibitem{JPS02}
T.~A. Johansen, I. Petersen, and O. Slupphaug,
\newblock ``Explicit sub-optimal linear quadratic regulation with state and
  input constraints,''
\newblock {\em Automatica J. IFAC}, vol. 38, no. 8, pp. 1099--1111, 2002.

\bibitem{SdDG00}
M.~M. {Seron}, J.~A. {De Dona}, and G.~C. {Goodwin},
\newblock ``Global analytical model predictive control with input
  constraints,''
\newblock in {\em Proceedings of the 39th IEEE Conference on Decision and
  Control}, Dec 2000, vol.~1, pp. 154--159.

\bibitem{SKJTJ06}
J. Spj{\o}tvold, E.~C. Kerrigan, C.~N. Jones, P. T{\o}ndel, and
  T.~A. Johansen,
\newblock ``On the facet-to-facet property of solutions to convex parametric
  quadratic programs,''
\newblock {\em Automatica}, vol. 42, no. 12, pp. 2209--2214, 2006.

\bibitem{PaSa10}
P.~Patrinos and H.~Sarimveis,
\newblock ``A new algorithm for solving convex parametric quadratic programs
  based on graphical derivatives of solution mappings,''
\newblock {\em Automatica}, vol. 46, no. 9, pp. 1405--1418, 2010.

\bibitem{ODP17}
R.~Oberdieck, N.A. Diangelakis, and E.N. Pistikopoulos,
\newblock ``Explicit model predictive control: A connected-graph approach,''
\newblock {\em Automatica}, vol. 76, pp. 103--112, 2017.

\bibitem{ZhXi18}
J.~Zhang and X. Xiu,
\newblock ``K-d tree based approach for point location problem in explicit
  model predictive control,''
\newblock {\em J. Franklin Inst.}, vol. 355, no. 13, pp. 5431--5451, 2018.

\bibitem{KJPHKB19}
M. Kvasnica, C.~N. Jones, I. Pejcic, J. Holaza, M. Korda, and
  P. Bakar{\'a}{\v c},
\newblock ``Real-time implementation of explicit model predictive control,''
\newblock in {\em Handbook of model predictive control}, Control Eng., pp.
  387--412. Birkh{\"a}user/Springer, Cham, 2019.

\bibitem{GBN11}
A.~Gupta, S.~Bhartiya, and P. S. V.~Nataraj,
\newblock ``A novel approach to multiparametric quadratic programming,''
\newblock {\em Automatica}, vol. 47, no. 9, pp. 2112--2117, 2011.

\bibitem{GoId83}
D.~Goldfarb and A.~Idnani,
\newblock ``A numerically stable dual method for solving strictly convex
  quadratic programs,''
\newblock {\em Math. Programming}, vol. 27, no. 1, pp. 1--33, 1983.

\bibitem{ScBi94}
C.~Schmid and T.~Biegler,
\newblock ``Quadratic programming methods for reduced {H}essian {SQP},''
\newblock {\em Computers and Chemical Engineering}, vol. 18, no. 9, pp.
  817--832, 1994.

\bibitem{Mon19}
M. M{\"o}nnigmann,
\newblock ``On the structure of the set of active sets in constrained linear
  quadratic regulation,''
\newblock {\em Automatica J. IFAC}, vol. 106, pp. 61--69, 2019.

\bibitem{MiMoe20}
R. Mitze and M. M{\"o}nnigmann,
\newblock ``A dynamic programming approach to solving constrained
  linear--quadratic optimal control problems,''
\newblock {\em Automatica}, vol. 120, pp. 109--132, 2020.

\bibitem{JPM17}
M. Jost, G. Pannocchia, and M. M\"{o}nnigmann,
\newblock ``Accelerating linear model predictive control by constraint
  removal,''
\newblock {\em Eur. J. Control}, vol. 35, pp. 42--49, 2017.

\bibitem{GrPaBook}
L. Gr\"{u}ne and J. Pannek,
\newblock {\em Nonlinear model predictive control},
\newblock Communications and Control Engineering Series. Springer, Cham, 2017,
\newblock Theory and algorithms, second edition.

\bibitem{AltmullerThesis}
N. Altm{\"u}ller,
\newblock {\em Model Predictive Control for Partial Differential Equations},
\newblock Ph.D. thesis, Bayreuth, December 2014.

\bibitem{DHK22MTNS}
S. Dubljevic, J.-P. Humaloja, and M. Kurula,
\newblock ``Explicit model predictive control for {PDE}s: {T}he case of a heat
  equation,''
\newblock {\em IFAC-PapersOnLine}, vol. 55, no. 30, pp. 460--465, 2022.

\bibitem{AK18}
B. Azmi and K. Kunisch,
\newblock ``Receding horizon control for the stabilization of the wave
  equation,''
\newblock {\em Discrete Contin. Dyn. Syst.}, vol. 38, no. 2, pp. 449--484,
  2018.

\bibitem{ItKu02}
K. Ito and K. Kunisch,
\newblock ``Receding horizon optimal control for infinite dimensional
  systems,''
\newblock vol.~8, pp. 741--760. 2002,
\newblock A tribute to J. L. Lions.

\bibitem{PGB14}
Van~Thang Pham, D. Georges, and G. Besan\c{c}on,
\newblock ``Infinite-dimensional predictive control for hyperbolic systems,''
\newblock {\em SIAM J. Control Optim.}, vol. 52, no. 6, pp. 3592--3617, 2014.

\bibitem{AzKu19}
B. Azmi and K. Kunisch,
\newblock ``A hybrid finite-dimensional {RHC} for stabilization of time-varying
  parabolic equations,''
\newblock {\em SIAM J. Control Optim.}, vol. 57, no. 5, pp. 3496--3526, 2019.

\bibitem{KuPf20}
K. Kunisch and L. Pfeiffer,
\newblock ``The effect of the terminal penalty in receding horizon control for
  a class of stabilization problems,''
\newblock {\em ESAIM Control Optim. Calc. Var.}, vol. 26, pp. Paper No. 58, 26,
  2020.

\bibitem{XuDu17}
Q. Xu and S. Dubljevic,
\newblock ``Linear model predictive control for transport-reaction processes,''
\newblock {\em AIChE Journal}, vol. 63, no. 7, pp. 2644 -- 2659, 2017.

\bibitem{DuHu20}
S. Dubljevic and J.-P. Humaloja,
\newblock ``Model predictive control for regular linear systems,''
\newblock {\em Automatica}, vol. 119, pp. 109--066, 2020.

\bibitem{ABMBook}
H. Attouch, G. Buttazzo, and G. Michaille,
\newblock {\em Variational analysis in {S}obolev and {BV} spaces}, vol.~17 of
  {\em MOS-SIAM Series on Optimization},
\newblock SIAM, second edition, 2014,
\newblock Applications to PDEs and optimization.

\bibitem{Zei85}
E. Zeidler,
\newblock {\em Nonlinear functional analysis and its applications. {III}},
\newblock Springer-Verlag, New York, 1985.

\bibitem{TaTr94}
R.~A. Tapia and M.~W. Trosset,
\newblock ``An extension of the {K}arush-{K}uhn-{T}ucker necessity conditions
  to infinite programming,''
\newblock {\em SIAM Rev.}, vol. 36, no. 1, pp. 1--17, 1994.

\bibitem{VilPhd}
J.~A. Villegas,
\newblock {\em A Port-Hamiltonian Approach to Distributed Parameter Systems},
\newblock Ph.D. thesis, University of Twente, Netherlands, 2007.

\bibitem{JacZwaBook}
B.~Jacob and H.~Zwart,
\newblock {\em Linear Port-{H}amiltonian Systems on Infinite-dimensional
  Spaces}, vol. 223 of {\em Operator Theory: Advances and Applications},
\newblock Birkh{\"a}user, 2012.

\bibitem{HaMa07}
V.~Havu and J.~Malinen,
\newblock ``The {C}ayley transform as a time discretization scheme,''
\newblock {\em Numer. Funct. Anal. Optim.}, vol. 28, no. 7-8, pp. 825--851,
  2007.

\bibitem{TuWeBook}
M. Tucsnak and G. Weiss,
\newblock {\em Observation and control for operator semigroups},
\newblock Birkh{\"a}user Advanced Texts. Birkh{\"a}user Verlag, 2009.

\bibitem{GSS20}
L. Gr\"{u}ne, M. Schaller, and A. Schiela,
\newblock ``Exponential sensitivity and turnpike analysis for linear quadratic
  optimal control of general evolution equations,''
\newblock {\em J. Differential Equations}, vol. 268, no. 12, pp. 7311--7341,
  2020.

\bibitem{BCK16}
J. Buerger, M. Cannon, and B. Kouvaritakis,
\newblock ``Active set solver for min-max robust control with state and input
  constraints,''
\newblock {\em Internat. J. Robust Nonlinear Control}, vol. 26, no. 15, pp.
  3209--3231, 2016.

\end{thebibliography}

\def\cprime{$'$} \def\cprime{$'$}

\end{document}